\documentclass{amsart}
\usepackage{graphics}
\usepackage{graphicx}
\usepackage{amsfonts,amsmath,mathrsfs}
\begin{document}

 \newtheorem{theorem}{Theorem}[section]
 \newtheorem{corollary}[theorem]{Corollary}
 \newtheorem{lemma}[theorem]{Lemma}{\rm}
 \newtheorem{proposition}[theorem]{Proposition}

 \newtheorem{defn}[theorem]{Definition}{\rm}
 \newtheorem{assumption}[theorem]{Assumption}
 \newtheorem{remark}[theorem]{Remark}
 \newtheorem{ex}{Example}
\numberwithin{equation}{section}
\def\la{\langle}
\def\ra{\rangle}
\def\glexe{\leq_{gl}\,}
\def\glex{<_{gl}\,}
\def\e{{\rm e}}

\def\fac{{\rm !}}
\def\x{\mathbf{x}}
\def\P{\mathbb{P}}
\def\S{\mathbf{S}}
\def\h{\mathbf{h}}
\def\m{\mathbf{m}}
\def\y{\mathbf{y}}
\def\bz{\mathbf{z}}
\def\F{\mathcal{F}}
\def\R{\mathbb{R}}
\def\T{\mathbf{T}}
\def\N{\mathbb{N}}
\def\D{\mathbf{D}}
\def\V{\mathbf{V}}
\def\U{\mathbf{U}}
\def\K{\mathbf{K}}
\def\Q{\mathbf{Q}}
\def\M{\mathbf{M}}
\def\oM{\overline{\mathbf{M}}}
\def\O{\mathbf{O}}
\def\C{\mathbb{C}}
\def\P{\mathbb{P}}
\def\Z{\mathbb{Z}}
\def\bZ{\mathbf{Z}}
\def\H{\mathcal{H}}
\def\A{\mathbf{A}}
\def\V{\mathbf{V}}
\def\AA{\overline{\mathbf{A}}}
\def\B{\mathbf{B}}
\def\c{\mathbf{C}}
\def\L{\mathcal{L}}
\def\bS{\mathbf{S}}
\def\H{\mathcal{H}}
\def\I{\mathbf{I}}
\def\Y{\mathbf{Y}}
\def\X{\mathbf{X}}
\def\cX{\mathbf{X}}
\def\G{\mathbf{G}}
\def\f{\mathbf{f}}
\def\z{\mathbf{z}}
\def\v{\mathbf{v}}
\def\y{\mathbf{y}}
\def\d{\hat{d}}
\def\x{\mathbf{x}}
\def\bI{\mathbf{I}}
\def\y{\mathbf{y}}
\def\g{\mathbf{g}}
\def\w{\mathbf{w}}
\def\b{\mathbf{b}}
\def\a{\mathbf{a}}
\def\p{\mathbf{p}}
\def\u{\mathbf{u}}
\def\bv{\mathbf{v}}
\def\q{\mathbf{q}}
\def\e{\mathbf{e}}
\def\s{\mathcal{S}}
\def\cc{\mathcal{C}}
\def\co{{\rm co}\,}
\def\tg{\tilde{g}}
\def\tx{\tilde{\x}}
\def\tg{\tilde{g}}
\def\tA{\tilde{\A}}
\def\bell{\boldsymbol{\ell}}
\def\bxi{\boldsymbol{\xi}}
\def\balpha{\boldsymbol{\alpha}}
\def\bbeta{\boldsymbol{\beta}}
\def\bgamma{\boldsymbol{\gamma}}
\def\bpsi{\boldsymbol{\psi}}
\def\supmu{{\rm supp}\,\mu}
\def\supp{{\rm supp}\,}
\def\cd{\mathcal{C}_d}
\def\cok{\mathcal{C}_{\K}}
\def\cop{COP}
\def\vol{{\rm vol}\,}
\def\om{\mathbf{\Omega}}
\def\f{\mathscr{F}}
\def\la{\langle\langle}
\def\ra{\rangle\rangle}
\def\blambda{\boldsymbol{\lambda}}
\def\btheta{\boldsymbol{\theta}}
\def\bphi{\boldsymbol{\phi}}
\def\bpsi{\boldsymbol{\psi}}
\def\bgamma{\boldsymbol{\gamma}}
\def\eeta{\boldsymbol{\eta}}
\def\bnu{\boldsymbol{\nu}}
\def\bom{\boldsymbol{\Omega}}
\def\fac{{\rm !}}
\def\tM{\hat{\M}}
\def\tv{\hat{\v}}

\title{A disintegration of the Christoffel function}
%\title[Classification]{On the Christoffel function and classification in data analysis}
\thanks{Research supported by the AI Interdisciplinary Institute ANITI  funding through the french program
\emph{Investing for the Future PI3A} under the grant agreement number ANR-19-PI3A-0004. The author is also affiliated with IPAL-CNRS laboratory, Singapore.}

\author{Jean B. Lasserre}
\address{LAAS-CNRS and Institute of Mathematics\\
University of Toulouse\\
LAAS, 7 avenue du Colonel Roche\\
31077 Toulouse C\'edex 4, France\\
Tel: +33561336415}
\email{lasserre@laas.fr}

\date{}

\begin{abstract} 
 We show that the Christoffel function (CF) factorizes (or can be disintegrated) as the product of two Christoffel functions, one associated with the marginal and the another related to the conditional distribution, in the spirit of 
 \emph{``the CF of the disintegration is the disintegration of the CFs"}. In the proof one uses an apparently overlooked 
  property (but interesting in its own) which states that any sum-of-squares polynomial is the Christoffel function of some linear form
 (with a representing measure in the univariate case). The same is true for the convex cone of polynomials that are positive
 on a basic semi-algebraic set. This interpretation of the CF establishes another bridge between polynomials optimization and orthogonal polynomials.
\end{abstract}
\maketitle
%% Abstract in the other language
%\begin{altabstract} 
% Nous montrons que la fonction de Christoffel (CF) se factorise en le produit de deux fonctions de Christoffel dont une est celle de la marginale et l'autre est li\'ee \`a la probabilit\'e conditionnelle. La d\'emonstration utilise une propri\'et\'e 
% apparemment ignor\'ee (mais int\'eressante en soi), qui stipule que tout polyn\^ome qui est somme de carr\'es est aussi la fonction de Christoffel  d'une forme lin\'eaire (repr\'esent\'ee par une mesure dans le cas univari\'e). The same is true for the convex cone of polynomials positive on a basic semi-algebraic set.
% Cette interpr\'etation de la fonction de Christoffel fournit un pont suppl\'ementaire 
% entre l'optimisation polynomiale et les polyn\^omes orthogonaux.
%\end{altabstract}

% Example of section
\section{Introduction}
It is well-known that a probability measure $\mu$ on a Cartesian product $\X\times\Y\subset\R^n\times\R^p$ of Borel spaces,
\emph{disintegrates} into $\hat{\mu}(d\y\,\vert\,\x)\,\phi(d\x)$ with 
its marginal $\phi(d\x)$ on $\X$ and its conditional probability $\hat{\mu}(dy\,\vert\,\x)$
on $\Y$, given $\x\in\X$. That is:
\begin{equation}
\label{eq-0}
\mu(\A\times\B)\,=\,\int_{\X\cap\A}\hat{\mu}(\B\,\vert\,\x)\,\phi(d\x)\,,\quad\forall \A\in\mathcal{B}(\R^n)\,,\B\in\mathcal{B}(\R^p)\,.\end{equation}
The goal of this note is to provide a similar disintegration (or factorization) for 
the family of its Christoffel functions $(\x,\y)\mapsto \Lambda^\mu_t(\x,\y)$, $t\in\N$.

\paragraph{Contribution} Our contribution is twofold. 

(i) Consider a probability measure $\mu$ on a compact subset 
$\bom:=\X\times \Y\subset \R^n\times \R$ and let $\Lambda^\mu_t:\R^{n+1}\to\R_+$ be its associated Christoffel function, i.e.,
with $\N^{n}_t=\{\,\bbeta\in\N^{n}: \vert\bbeta\vert\leq t\,\}$, 
\[(\x,y)\,\mapsto\quad \Lambda^\mu_t(\x,y)\,:=\,\sum_{(\balpha,j)\in\N^{n+1}_t}P_{\balpha,j}(\x,y)^2\,,\quad\forall (\x,y)\in\R^n\times\R\,,\]
where $(P_{\balpha,j})_{(\balpha,j)\in\N^{n+1}}\subset\R[\x,y]$ is a family of orthonormal polynomials with respect to (w.r.t.) $\mu$.
%In fact for reasons that will be transparent later, we instead prefer to consider its slight variant %$\hat{\Lambda}^\mu_t$ defined by
%\[(\x,y)\,\mapsto\quad \hat{\Lambda}^\mu_t(\x,y)\,:=\,\sum_{(\balpha,j)\in\N^{n}_t\times\N_t}P_{\balpha,j}(\x,y)^2\,,\quad\forall (\x,y)\in\R^n\times\R\,,\]
%that is, in the above sum we now consider all orthonormal polynomials $(P_{\balpha,j})\subset\R[\x,y]_{t,t}$ (instead of $\R[\x,y]_t$), where
%$\R[\x,y]_{t,t}$ denotes the space of polynomials of total degree at most $t$ in $\x$, and total degree at most $t$ in $y$.

Our main result states that $\Lambda^\mu_t$ \emph{disintegrates} (or \emph{factorizes}) into 
\begin{equation}
\label{eq-1}
\Lambda^\mu_t(\x,y)\,=\,\Lambda^{\phi}_t(\x)\cdot\Lambda^{\hat{\nu}_{\x,t}}_t(y)\,,\quad\forall (\x,y)\in \R^{n+1}\,,\:\forall t\in\N\,,
\end{equation}
where $\Lambda^{\phi}_t$ (resp. $\Lambda^{\hat{\nu}_{\x,t}}_t$) is the Christoffel function of the marginal $\phi$ of $\mu$ on $\X$
(resp. of some probability measure $\nu_{\x,t}$ on $\R$, given $\x\in\X$). 
Moreover, for every fixed $\x\in\X$,
one can compute explicitly the Hankel moment matrix of the measure $\nu_{\x,t}$ by solving a single convex optimization problem on positive definite matrices 
with $\log\,\mathrm{det}(\cdot)$ as objective function.

Notice how \eqref{eq-1} mimics the disintegration \eqref{eq-0}. 
Indeed, as we should expect from the disintegration \eqref{eq-1}, it turns out that for each fixed $\x\in\X$, the family $(\Lambda^{\nu_{\x,t}}_t)_{t\in\N}$ 
shares asymptotic properties of the Christoffel function $\Lambda^{\hat{\mu}}_t(y)$
of  the conditional probability $\hat{\mu}(dy\,\vert\,\x)$ on $\Y$, given $\x\in\X$.

Actually, the same disintegration \eqref{eq-1} holds if the conditioning is multivariate, i.e.,
on $\y\in\R^p$ given $\x\in\R^n$, with $p>1$. The only difference is that now 
$\nu_{\x,t}$ is a linear functional on $\R[\y]_t$
not necessarily represented by a probability measure on $\R^p$.

(ii) Interestingly, the technique of proof relies on a certain one-to-one mapping between interiors of the convex cone of sum-of-squares polynomials and its dual cone of moment matrices.  In particular, and as a by-product, it implies the following simple but apparently unnoticed result that every sum-of-squares polynomial is the reciprocal of a Christoffel function of some linear functional (guaranteed to have a representing measure in the univariate case).

\section{Notation, definitions and preliminary results}
\subsection{Notation and definitions}
Let $\R[\x]$ denote the ring of real polynomials in the variables $\x=(x_1,\ldots,x_n)$ and $\R[\x]_t\subset\R[\x]$ be its subset 
of polynomials of total degree at most $t$. Let $\N^n_t:=\{\balpha\in\N^n:\vert\balpha\vert\leq t\}$
(where $\vert\balpha\vert=\sum_i\alpha_i$) with cardinal $s_n(t)={n+t\choose n}$. Let $\bv_t(\x)=(\x^{\balpha})_{\balpha\in\N^n_t}$ 
be the vector of monomials up to degree $t$. 
%We also denote by $\R[\x,\y]_{t,t}\subset\R[\x,\y]_t$ the space of polynomials of degree at most $t$ in the variables $\x$, and of degree at most in the variables $\y$. 

Let $\Sigma[\x]_t\subset\R[\x]_{2t}$ be the convex cone of polynomials of total degree at most $2t$ which are sum-of-squares (in short SOS).
For a real symmetric matrix 
$\A=\A^T$ the notation $\A\succeq0$ (resp. $\A\succ0$) stands for $\A$ is positive semidefinite (p.s.d.) (resp. positive definite (p.d.)).
The support of a Borel measure $\mu$ on $\R^n$ is the smallest closed set $A$ such that
$\mu(\R^n\setminus A)=0$, and such a  set $A$ is unique.

\paragraph{Riesz functional} 
With a real sequence $\bphi=(\phi_{\balpha})_{\balpha\in\N^n}$ is associated 
the \emph{Riesz} linear functional
 $L_{\bphi}:\R[\x]\to\R$ defined by:
\[p\:(=\sum_{\balpha\in\N^n} p_{\balpha}\,\x^{\balpha}\,)\quad \mapsto L_{\bphi}(p)\,:=\,\sum_{\balpha\in\N^n}p_{\balpha}\,\phi_{\balpha}\,,\quad \forall p\in\R[\x]\,.\]
A sequence $\bphi=(\phi_{\balpha})_{\balpha}$ has a \emph{representing measure} if and only if there exists a  Borel measure $\phi$ on $\R^n$ such that $\int \x^{\balpha}\,d\phi=\phi_{\balpha}$, for all $\balpha\in\N^n$.

\paragraph{Moment matrix}
%Let $\bphi$ be a real sequenceBorel measure whose support $\bom\subset\R^n$ is compact with nonempty interior. Its 
With a real sequence $\bphi=(\phi_{\balpha})_{\balpha\in\N^n}$ is associated 
its \emph{moment matrix} $\M_t(\bphi)$ of order (or degree) $t$. It is  a
real symmetric matrix 
with rows and columns indexed by $\N^n_t$, and with entries
\[\M_t(\bphi)(\balpha,\bbeta)\,:=\,L_{\bphi}(\x^{\balpha+\bbeta})\,=\,\phi_{\balpha+\bbeta}\,,\quad\balpha,\bbeta\in\N^n_t\,.\]
Importantly, $\M_t(\bphi)$ depends only on moments $\phi_{\balpha}$ with $\vert\balpha\vert\leq 2t$.
If $\bphi$ has a representing measure $\phi$ then we also write $\M_t(\phi)$ and
necessarily $\M_t(\phi)$ is p.s.d. for all $t$, i.e., $\M_t(\phi)\succeq0$ for all $t$.

\paragraph{Christoffel function} 
Let $\bphi=(\phi_{\balpha})_{\balpha\in\N^n}$ be such that $\M_t(\bphi)\succ0$ for all $t$, and
let $(P_{\balpha})_{\balpha\in\N^n}\subset\R[\x]$ be a family of polynomials, orthonormal with respect to $\bphi$, i.e.,
%\[\int_{\bom}P_{\balpha}\,P_{\bbeta}\,d\phi\,=\,\delta_{\balpha=\bbeta}\,,\quad\forall \balpha,\bbeta\in\N^n\,.\]
\[L_{\bphi}(P_{\balpha}\,P_{\bbeta})\,=\,\delta_{\balpha=\bbeta}\,,\quad\forall \balpha,\bbeta\in\N^n\,.\]
Then the 
Christoffel function (CF) $\Lambda^{\bphi}_t:\R^n\to\R_+$  associated with $\bphi$, is defined by
\begin{equation}
\label{def-christo-1}
\x\mapsto \Lambda^{\bphi}_t(\x)\,:=\,\left[\sum_{\balpha\in\N^n_t} P_{\balpha}(\x)^2\right]^{-1}\,,\quad\forall \x\in\R^n\,,\end{equation}
and recalling that $\M_t(\bphi)$ is nonsingular, it turns out that 
\begin{equation}
\label{def-christo-11}
\Lambda^{\bphi}_t(\x)\,=\,\left[\,\v_t(\x)^T\,\M_t(\bphi)^{-1}\,\v_t(\x)\,\right]^{-1}\,,\quad\forall \x\in\R^n\,.
\end{equation}
An equivalent and variational definition is also
\begin{equation}
\label{def-christo-2}
%\Lambda^\phi_t(\x)\,=\,\inf_{p\in\R[\x]_t}\{\,\int_{\bom} p^2\,d\phi\::\: p(\x)\,=\,1\,\}\, ,\quad\forall \x\in\R^n\,.\end{equation}
\Lambda^{\bphi}_t(\x)\,=\,\inf_{p\in\R[\x]_t}\{\,L_{\bphi}(p^2)\::\: p(\x)\,=\,1\,\}\, ,\quad\forall \x\in\R^n\,.\end{equation}
In \cite{annals-prob} the authors describe a way to obtain a family of orthonormal
polynomials w.r.t. $\bphi$ from the moment matrices $\M_t(\bphi)\succ0$ via simple determinant calculations. We will use this construction with a special ordering of the monomials 
that index the rows and columns of $\M_t(\bphi)$.

If $\bphi$ has a representing measure $\phi$ we also write its CF as $\Lambda^\phi_t$. The CF is usually defined  for measures $\phi$ on a compact set $\bom$ rather than for 
linear functionals $\bphi$ with $\M_t(\bphi)\succ0$ for all $t$. In this case 
one interesting and distinguishing feature of the CF is that as $t$ increases, $\Lambda^\phi_t(\x)\downarrow 0$ exponentially fast for every $\x$ outside the support of $\phi$. In other words, 
$\Lambda_t^\phi$ identifies the support of $\phi$ when $t$ is sufficiently large, a nice property that can be exploited for outlier 
detection in some data analysis applications; see for instance \cite{CD-2022,adv-comp}. 
In addition, at least in dimension $n=2$ or $n=3$, one may visualize this property even for small $t$, as the resulting superlevel sets 
$\bom_\gamma:=\{\,\x: \Lambda^\phi_t(\x)\geq \gamma\,\}$, $\gamma\in\R$, capture the shape of $\bom$ quite well; see e.g. \cite{neurips}.

\subsection{A specific family of orthonormal polynomials}
\label{variant}
Let $\mu$ be a Borel measure on a compact subset of $\X\times\Y\subset\R^n\times\R^p$,
and let $\M_t(\mu)$ be the moment matrix of $\mu$ with rows and columns indexed by
the monomials $(\x^{\balpha}\,\y^{\bbeta})_{(\balpha,\bbeta)\in\N^{n+p}_t}$ listed 
according to some ordering noted ``$\preceq$" between monomials, defined as follows.
First in the list, we find all monomials $(\x^{\balpha})_{\balpha\in\N^n_t}$ (i.e. all monomials
$\x^{\balpha}\,\y^{\bbeta}$ with $\vert\bbeta\vert=0$) listed e.g. according to the lexicographic ordering. Then we find all monomials $\x^{\balpha}\,\y^{\bbeta}$ with $\vert\bbeta\vert=1$,
then followed  by monomials $\x^{\balpha}\,\y^{\bbeta}$ with $\vert\bbeta\vert=2$, etc.
Below is displayed $\M_2(\mu)$ in the bivariate case $(n,p)=(1,1)$.
\begin{equation}
\label{ex1}
\M_2(\mu)\,=\,
\left[\begin{array}{cccccc}
\mu_{0,0}&\mu_{1,0}&\mu_{2,0} &\mu_{0,1}&\mu_{1,1}&\mu_{0,2}\\
\mu_{1,0}&\mu_{2,0}&\mu_{3,0} &\mu_{1,1} &\mu_{2,1}&\mu_{1,2}\\
\mu_{2,0}&\mu_{3,0}&\mu_{4,0}&\mu_{2,1}&\mu_{3,1}&\mu_{2,2}\\
\mu_{0,1}&\mu_{1,1}&\mu_{2,1}&\mu_{0,2}&\mu_{1,2}&\mu_{0,3}\\
\mu_{1,1}&\mu_{2,1}&\mu_{3,1}&\mu_{1,2}&\mu_{2,2}&\mu_{1,3}\\
\mu_{0,2}&\mu_{1,2}&\mu_{2,2}&\mu_{0,3}&\mu_{1,3}&\mu_{0,4}
\end{array}\right]\,.\end{equation}
Similarly, let $\v_t(\x,y)$ be the vector of monomials that form a basis of $\R[\x,\y]_{t}$ 
listed with the same above ordering  ``$\preceq$"; for instance with $(n,p)=(1,1)$ and $t=2$,
$\v_2(x,y)=(1,x,x^2,y,xy,y^2)$.
Then by \eqref{def-christo-11}, the Christoffel function $\Lambda^\mu_t$ is given by 
\begin{equation}
 \label{eq-4}
 (\x,\y)\,\mapsto\quad \Lambda^\mu_t(\x,\y)\,:=\,\v_t(\x,\y)^T\,\M_t(\mu)^{-1}\,\v_t(\x,\y)\,,\quad\,\forall\,(\x,\y)\in\R^n\times\R^p\,.
\end{equation}
%The difference between $\hat{\Lambda}^\mu_t$ and $\Lambda^\mu_t$ is that
%the moment matrix $\hat{\M}_t(\mu) $ contains moments up to order $4t$ whereas $\M_t(\mu) $ contains moments up to order $2t$ only. Equivalently
%\[\hat{\Lambda}^\mu_t(\x,\y)\,=\,\displaystyle\min_{p\in\L_t}\,\{\,\int p^2\,d\mu: p(\x,y)\,=\,1\,\}\,,\]
%where $\L_t=\R[\x,\y]_{t,t}$ understood as a subspace of $L^2(\mu)$ of dimension ${n+t\choose t}{p+t\choose p}$
%(instead of $V_t=\R[\x,\y]_t$ of dimension ${n+p+t\choose t}$ for $\Lambda^\mu_t$.
%
With next see that with ordering ``$\preceq$" defined above,
one we may define a certain family of orthonormal polynomials $(P_{\balpha,\bbeta})\subset\R[\x,\y]_{t}$ by following the recipe
described in \cite{annals-prob} and  that we briefly summarize:
To compute $P_{\balpha,\bbeta}\in\R[\x,\y]_{t}$ one proceeds in three steps:
\begin{itemize}
\item From $\M_t(\mu)$ extract its submatrix $S(\balpha,\bbeta)$ with rows and columns indexed by $(\bgamma,\eeta)\preceq (\balpha,\bbeta)$.
\item Delete the last row and replace it with the monomials $(\x^{\bgamma}\,\y^{\eeta})$ with $(\bgamma,\eeta)\preceq (\balpha,\bbeta)$.
\item Define $\tilde{P}_{\balpha,\bbeta}(\x,\y):=\mathrm{det}(S(\balpha,\bbeta))$ and then normalize 
$P_{\balpha,\bbeta}=\tau\,\tilde{P}_{\balpha,\bbeta}$ with $\tau>0$ such that
$\tau^2\int (\tilde{P}_{\balpha,\bbeta})^2\,d\mu=1$.
\end{itemize}
\begin{lemma}
\label{lem-2}
With the ordering ``$\preceq$" and the above construction, the orthonormal polynomials $(P_{\balpha,0})_{\balpha\in\N^n_t}$ 
depend only on $\x$, and are orthonormal w.r.t. the marginal $\phi$ of $\mu$.%, up to degree $t$.
\end{lemma}
\begin{proof}
 In the above construction the orthonormal polynomials $(P_{\balpha,0})_{\balpha\in\N^n_t}$ are obtained from the submatrices $S(\balpha,0)$ of $\M_t(\mu)$, $\balpha\in\N^n_t$,
 which are exactly the submatrices of $\M_t(\phi)$ since they are formed with only monomials $\x^{\balpha}$ (as $\vert\bbeta\vert=0$).
  Hence the conclusion follows.
\end{proof}
For illustration purpose, with $\M_2(\mu)$ as in \eqref{ex1},
$\tilde{P}_{0,0}(x,y)=\mu_{0,0}$, and
\[\tilde{P}_{1,0}(x,y)=\mathrm{det}\left(\left[\begin{array}{cc}
\mu_{0,0}&\mu_{1,0}\\
1& x\end{array}\right]\right)\,;\quad
\tilde{P}_{2,0}(x,y)\,=\,\mathrm{det}\left[\begin{array}{ccc}
\mu_{0,0}&\mu_{1,0}&\mu_{2,0}\\
\mu_{1,0}&\mu_{2,0}&\mu_{3,0}\\
1& x&x^2\end{array}\right]\,,\]
so that $\tilde{P}_{1, 0}(x,y)=\mu_{0,0}\,x-\mu_{1,0}$, and
\[\Tilde{P}_{2,0}(x,y)=(\mu_{0,0}\mu_{2,0}-\mu_{1,0}^2)x^2-(\mu_{0,0}\mu_{3,0}-\mu_{1,0}\mu_{2,0})\,x
%&&
+(\mu_{1,0}\mu_{3,0}-\mu_{2,0}^2)\,.\]
\begin{corollary}
\label{coro-1}
Let $\mu$ be a Borel measure on $\X\times \Y$ with marginal $\phi$ on $\X$, and assume that $\M_t(\mu)\succ0$ for all $t$.
Let $(P_{\balpha,\bbeta})$, $(\balpha,\bbeta)\in\N^{n+p}_t$, be the family or orthonormal polynomials defined in Section \ref{variant}, and let
$\Lambda^\mu_t$ be as in \eqref{eq-4}. Then:
\begin{equation}
\label{coro-1-eq-1}
 \Lambda^{\mu}_t(\x,\y)^{-1}\,=\,\Lambda^{\phi}_t(\x)^{-1}+
 \sum_{(\balpha,\bbeta)\in\N^{n+p}_t,\vert\bbeta\vert\geq 1}P_{\balpha,\bbeta}(\x,\y)^2\,.
\end{equation}
\end{corollary}
\begin{proof}
 By Lemma \ref{lem-2}, the polynomials $(P_{\balpha,0})_{\balpha\in\N^n_t}$ depend on $\x$ only and are orthonormal w.r.t. $\phi$. Therefore
 by \eqref{def-christo-1}:
  \[\sum_{\balpha\in\N^n_t}P_{\balpha,0}(\x)^2\,=\,\Lambda^\phi_t(\x)^{-1}\,,\quad\forall \x\in\R^n\,,\quad t\in\N\,.\]
 Then the result follows from 
 \begin{eqnarray*}
 \Lambda^\mu_t(\x,y)^{-1}&=&\sum_{(\balpha,\bbeta)\in\N^{n+p}_t}P_{\balpha,\bbeta}(\x,\y)^2\\
 &=&\sum_{\balpha\in\N^n_t,\bbeta=0}P_{\balpha,0}(\x,\y)^2
 +\sum_{(\balpha,\bbeta)\in\N^{n+p}_t:\vert\bbeta\vert\geq 1}P_{\balpha,\bbeta}(\x,\y)^2\\
 &=&\Lambda^{\phi}_t(\x)^{-1}+
 \sum_{(\balpha,\bbeta)\in\N^{n+p}_t:\vert\bbeta\vert\geq 1}P_{\balpha,\bbeta}(\x,\y)^2\,.
  \end{eqnarray*}
\end{proof}

\subsection{Positive polynomials and Christoffel functions}

Recall that $p\in\Sigma[\x]_{t}$ (i.e., $p$ is an SOS of degree at most $2t$)
  if and only if
there exists a real symmetric matrix $\Q\succeq0$ such that $p(\x)\,=\,\v_t(\x)^T\,\Q\,\v_t(\x)$ for all $\x\in\R^n$. Notice that except when $t=1$, there are several possible choices for $\Q$ which is called a Gram matrix of $p$.  As we next see, one choice is particularly interesting. 
The dual cone $\Sigma^*_t$ is the convex cone characterized by:
\[\Sigma^*_t\,=\,\{\,\bphi\in\N^n_{2t}: \M_t(\bphi)\,\succeq\,0\,\}\,.\]
\begin{lemma}
\label{nesterov}
 (i)  Every  SOS polynomial in the interior of $\Sigma_t$ is the reciprocal of the Christoffel function 
 $\Lambda^{\bphi}_t$ of some linear functional  $L_{\bphi}$, with $\bphi\in\mathrm{int}(\Sigma^*_t)$.
 That is, $p\in \mathrm{int}(\Sigma_t)$ if and only if 
 \[p(\x)=\v_t(\x)^T\,\M_t(\bphi)^{-1}\v_t(\x)\,,\quad\forall \x\in\R^n\,,\]
  for some moment sequence $\bphi\in\mathrm{int}(\Sigma^*_t)$.  Therefore $p^{-1}$ is the Christoffel function associated with some linear functional $\bphi$ (with not necessarily a representing measure $\phi$).
   
 (ii) In addition, if $p\in \mathrm{int}(\Sigma_t)$  is univariate then $\bphi$ has a representing 
 measure $\phi$, and so $p^{-1}$ is the Christoffel function $\Lambda^\phi_t$ of some
 measure $\phi$ on the real line.
\end{lemma}
\begin{proof}
 The first part of the statement is a direct consequence from
 Nesterov \cite[Theorem 17.3, p. 412]{nesterov} which states that the respective interiors 
 of $\Sigma_t$ and its dual$\Sigma^*_t$ are in one-to-one correspondence, and $-\log\mathrm{det}(\A)$ 
 is a $\tau$-self-concordant barrier function associated with the convex cone $\Sigma_t$, with $\tau={n+t\choose t}$. The second statement
 follows from the characterization \eqref{def-christo-11} of $\Lambda^{\bphi}_t$.
\end{proof}
Surprisingly, the fact that every (strictly positive) SOS polynomial of degree at most $2t$, is the Christoffel function $\Lambda^{\bphi}_t$ of some linear 
functional $L_{\bphi}$ on $\R[\x]_{2t}$ with $\M_t(\bphi)\succ0$, does not seem 
to have been noticed before, even though Nesterov's result \cite[Theorem 17.3]{nesterov} is quite classical in convex conic optimization.
In addition, observe that Lemma \ref{nesterov} is the degree-$t$ analogue of the well-known fact that 
the Gram matrix of every positive quadratic form is the covariance of a Gaussian measure (possibly after scaling). 
Finally, and said differently, the Christoffel functions $\Lambda^{\bphi}_t$
associated with moment matrices $\M_t(\bphi)$ of size $\tau={n+t\choose t}$, 
encode the \emph{central path}\footnote{In convex optimization, the central path 
associated with a convex cone $\K$, plays a central role in the analysis of the computational complexity of interior points methods for optimizing over such a cone.} of the convex cone $\Sigma_t$ of $n$-variate SOS polynomials of degree $2t$.
 
 We even have a similar result in a more general context. With $g_j\in\R[\x]_{d_j}$,  let $\g_j=(g_{j\balpha})_{\balpha\in\N^n_{d_j}}$ denote its vector of coefficients,  $j=0,\ldots,m$.
 Given  a real sequence $\bphi=(\phi_{\balpha})_{\balpha\in\N^n}$, define
 the new sequences $g_j\cdot\bphi\,:=\,(\phi_{j,\balpha})_{\balpha\in\N^n}$, where for each $j=0,\ldots,m$, 
 \[\phi_{j,\balpha}\,:=\,\sum_{\bbeta\in\N^n_{d_j}} g_{j\bbeta}\,\phi_{\balpha+\bbeta}\,,\quad \forall\balpha\in\N^n\,.\]
 
 \begin{lemma}
 \label{lemma-multi}
 With $g_j\in \R[\x]_{d_j}$, let $s_j:=\lceil \mathrm{deg}(g_j)/2\rceil$, $j=0,\ldots,m$, 
 and for every $t\geq\max_j s_j$, let $\K_t\subset\R[\x]$ be the convex cone defined by:
\begin{equation}
\label{module}
\K_t\,:=\,\{\sum_{j=0}^m \sigma_j(\x)\,g_j(\x)\::\quad \sigma_j\in\Sigma_{t-s_j}\,\}\,\subset\R[\x]_{2t}\,.\end{equation}
If $p\in\mathrm{int}(\K_t)$ then $p\geq0$ on $S:=\{\,\x: g_j(\x)\geq0\,,\:j=0,\ldots,m\,\}\subset\R^n$ and
\begin{equation}
 p(\x)\,=\,\sum_{j=0}^m\Lambda^{g_j\cdot\bphi}_{t-s_j}(\x)^{ -1}\,g_j(\x),\quad\forall \x\in\R^n\,
 \end{equation}
 for some linear functional ${\bphi}=(\phi_{\balpha})_{\balpha\in\N^n_{2t}}$ that satisfies
 $\M_{t-s_j}(g_j\cdot\bphi)\succ0$ for all $j=0,\ldots,m$.
 \end{lemma}
\begin{proof}
 The proof is in the same spirit and again relies on the one-to-one mapping between the interior of
 the  convex cone $\K_t$ and that of its dual 
\[\K^*_t\,=\,\,\{\,\bphi\in\N^n_{2t}:\: \M_{t-s_j}(g_j\cdot\bphi)\,\succeq0\,,\quad j=0,\ldots,m\,\}\,,\]
described in \cite[Theorem 17.6(2), p. 416]{nesterov}. 
Translated in our notation, (17.9) in Nesterov \cite[Theorem 17.7, p. 417]{nesterov} reads,
\[p(\x)\,=\,\sum_{j=0}^m \v_{t-s_j}(\x)^T\M_{t-s_j}(g_j\cdot\bphi)^{-1}\,\v_{t-s_j}(\x)\,,\quad\forall \x\in\R^n\,.\]
for some $\bphi\in\mathrm{int}(\K^*_t)$.
\end{proof}
Again the Christoffel functions $\Lambda^{g_j\cdot\bphi}_{t-s_j}$ associated with
the moment matrices $\M_{t-s_j}(g_j\cdot\bphi)$ encode the central path of the convex cone
$\K_t$ in \eqref{module}.
For compact set $S$ (with an additional Archimedean assumption), 
the cone $\K_t$ is very important in the Moment-SOS hierarchy for polynomial optimization \cite{CUP}. 
It is used to replace the intractable positivity constraint ``$p\geq0$" on $S$,
with the more restrictive constraint ``$p\in\K_t$" (and let $t$ increase) because the latter being semidefinite representable,
is tractable. 
%Hence Lemma \ref{lemma-multi} provides the additional and potentially useful information that for every
%$p\in \mathrm{int}(\K_t)$,there is a particular choice for SOS weights $\sigma_j$ in \eqref{module}, 
%namely $\sigma_j^{-1}=\Lambda^{g_j\cdot\phi}_{t-s_j}$, $j=0,\ldots,m$. That is, 
%the $\sigma_j$'s with $j>0$, can all be deduced from $\sigma_0$.

\section{Main result}

Let $\mu$ be a Borel measure on a compact set $\bom\subset\X\times\Y\subset\R^n\times\R$ which disintegrates into 
its marginal $\phi$ on $\X\subset\R^n$ and its conditional probability $\hat{\mu}(dy\,\vert\,\x)$ on $\Y_\x\subset\Y$ for every $\x\in\X$.
Throughout the rest of the paper we assume that $\bom$ has nonempty interior so that $\M_t(\mu)\succ0$ for all $t\in\N$, where $\hat{\mu}_t(\mu)$ is constructed as in Section \ref{variant}.
\begin{theorem}
\label{th-main}
 %Assume that $\M_t(\phi)$ and $\M_t(\hat{\mu})$ are non-singular for every $t$, so that $\M_t(\phi)\,,\:\M_t(\hat{\mu})\succ0$ for all $t$. 
 Let $\Lambda^\mu_t$ be as in \eqref{eq-4} with $\M_t(\mu)$ constructed as indicated just above \eqref{eq-4}. Then for every $\x\in\X$ and $t\in\N$, there exists a probability measure $\nu_{\x,t}$ on $\R$ such that
  \begin{equation}
  \label{th-main-1}
 \Lambda^\mu_t(\x,y)\,=\,\Lambda^\phi_t(\x)\cdot\Lambda^{\nu_{\x,t}}_t(y)\,,\quad\forall \x\in\X\,,\:y\in\R\,.
\end{equation}
 \end{theorem}
\begin{proof}
Let $t\in\N$ and $\x\in\R^n$ be fixed.
 From \eqref{coro-1-eq-1} in Corollary \ref{coro-1} and as $p=1$,
 \begin{eqnarray*}
 \frac{\Lambda^\mu_t(\x,y)^{-1}}{\Lambda^\phi_t(\x)^{-1}}&=&
 1+\Lambda^{\phi}_t(\x)\,\left[
 \sum_{(\balpha,j)\in\N^{n+1}_t,1\leq j\leq t }P_{\balpha,j}(\x,y)^2\,\right]\\
 &=:&p_t(y\,;\,\x)\,\in\,\R[y]_{2t}\,.
 \end{eqnarray*}
 Hence for each fixed $\x\in\R^n$, $1\leq p_t(y\,;\,\x)\in\R[y]$ is a 
 strictly positive univariate SOS. Therefore by Lemma \ref{nesterov}(ii) there exists a Borel measure 
 $\nu_{\x,t}$ on $\R$ such that $p_t(y\,;\,\x)^{-1}\,=\,\Lambda_t^{\nu_{\x,t}}(y)$, which yields 
 \eqref{th-main-1}. 
 %It remains to prove that the measure $\nu_{\x,t}$ is the same as $\nu_{\x,t'}$ whenever $t'>t$.
\end{proof}
When $\x\in\X$, notice how well \eqref{th-main-1} mimics the disintegration \eqref{eq-0} of $\mu$
into its marginal $\phi$ on $\X$ and its conditional $\hat{\mu}(dy\,\vert\,\x)$ on $\Y_\x$, 
given $\x\in\X$. However when $\x\in\X$, it remains to relate the family of measures $(\nu_{\x,t})_{t\in\N}$ on $\Y_\x$ with the conditional probability $\hat{\mu}(dy\,\vert\,\x)$. 

\paragraph{Computing the moment matrix of $\nu_{\x,t}$}
To obtain the moment matrix of $\nu_{\x,t}$, for an arbitrary but fixed $\x\in\R^n$, is relatively easy. 
Let $S_t$ be the space of $(t+1)\times (t+1)$ real symmetric matrices.\\

- Compute the polynomial $y\mapsto p_t(y;\x):=\Lambda^\mu_t(\x,y)^{-1}/\Lambda^{\phi}_t(\x)^{-1}$ which is an SOS in "$y$" of degree $2t$. This is easy once moments of $\mu$ are available. Indeed one computes $\Lambda^\mu_t(\x,y)$ (resp. $\Lambda_t^\phi(\x)$) via \eqref{def-christo-11} with the moment matrix $\M_t(\mu)$ and $\v_t(\x,\y)$
(resp. $\M_t(\phi)$ and $\v_t(\x)$).

- Then following \cite[p. 412]{nesterov}, solve the convex optimization problem 
\begin{equation}
\label{pb-optim}
\M_t(\nu_{\x,t})\,=\,\arg\min_{0\prec\Q\in\,S_t}\,\{\,-\log\,\mathrm{det}(\Q)
\,:\: p_t(y;\x)\,=\,\v_t(y)^T\,\Q\,\v_t(y)\,,\quad\forall y\,\}.\end{equation}
The optimization problem \eqref{pb-optim} is convex and can be solved by 
off-the-shelf solvers like e.g. CVX \cite{cvx}.

\paragraph{Multivariate conditional}

If $p>1$ and $\Y\subset\R^p$, then we still obtain the decomposition \eqref{th-main-1}
\begin{equation}
\label{multi-p}
\Lambda^\mu_t(\x,\y)\,=\,\Lambda^\phi_t(\x)\cdot\Lambda^{\bnu_{\x,t}}_t(\y)\,,\quad\forall \x\in\X\,,\:\y\in\R^p\,,\end{equation}
with exactly the same proof as that of Theorem \ref{th-main}. The difference with \eqref{th-main-1} is that $\bnu_{\x,t}$  in \eqref{multi-p} 
is a linear functional  on $\R[\y]_t$ which is not guaranteed to
have  a representing measure $\nu_{\x,t}$ on $\R^p$.

\subsection{Discussion}
Define the scalar 
$s_n(t):={n+t\choose t}$ for every integer $t,n$.
Under some conditions on the sets $\bom$ and $\X,\Y_\x$ and if $\mu$ has a density w.r.t. Lebesgue measure on $\bom$ that also satisfies some conditions, then one may indeed relate the family
$(\nu_{\x,t})_{t\in\N}$ with the conditional probability $\hat{\mu}(dy\,\vert\,\x)$ on$\Y$, given $\x\in\X$. Under such conditions one may interpret the limit $s_{n+1}(t)\Lambda^\mu_t(\x,y)$ and $s_n(t)\Lambda^{\phi}_t(\x)$,
as $t$ increases, in terms of the density of $\mu$ and an \emph{equilibrium} measure intrinsically related to the respective supports $\bom$ and $\X$. 
For such conditions the interested reader is referred to \cite{CD-2022,Lubinsky} and the many references therein. For instance:
\begin{corollary} \label{outside}

Let $\bom=\X\times\Y\subset\R^{n+1}$ be compact with 
$\bom=\overline{\mathrm{int}(\bom)}$, $\X=\overline{\mathrm{int}(\X)}$, 
and assume that $\mu$ has a density 
$f$ w.r.t. Lebesgue on $\R^{n+1}$, bounded away from $0$ on $\bom$. 

If $\x\in\mathrm{int}(\X)$ but $(\x,y)\not\in\bom$, then as $t$ increases, $\Lambda^{\nu_{\x,t}}_t(y)\downarrow 0$ exponentially fast (as would do
the Christoffel function $\Lambda^{\hat{\mu}}_t(y)$ of the conditional probability $\hat{\mu}(dy\,\vert\,\x)$).
%(ii) If $(\x,y)\in\mathrm{int}(\bom)$ then $\Lambda^{\nu_{\x,t}}_t(y)$ grows as $O(t)$.
\end{corollary}
\begin{proof}
 By \cite{adv-comp,CD-2022}, as $(\x,y)\not\in\bom$, $\Lambda^{\mu}_t(\x,y)\downarrow 0$ exponentially fast as $t$ increases.
 On the other hand, as $\x\in\mathrm{int}(\X)$ and the density of $\phi$ w.r.t. Lebesgue on $\R^n$ 
is bounded away from zero, $\Lambda^{\phi}_t(\x)^{-1}$ increases with $t$ not faster than $O(t^n)$.
 Therefore by \eqref{th-main-1}, $\Lambda^{\nu_{\x,t}}_t(y)^{-1}$ has to grow exponentially fast with $t$. The same conclusion holds for $\hat{\mu}(dy\,\vert\,\x)$; indeed let $y$ be outside the support $\Y_\x$
 of $\hat{\mu}(dy\,\vert\,\x)$. The density of $\hat{\mu}(dy\,\vert\,\x)$ which reads $y\mapsto f(\x,y)/{\int_\Y f(\x,y)\hat{\mu}(dy\,\vert\,\x)}$ on $\Y_\x$, is bounded away from zero. Therefore $\Lambda^{\hat{\mu}}_t(y)\downarrow 0$ exponentially fast as $t$ increases.
 % (ii) Again by \cite{CD-2022} $\Lambda^{\mu}_t(\x,y)^{-1}$ grows as $O(t^{n+1})$
\end{proof}
So Corollary \ref{outside} states that whenever $\x\in\X$ and $y\not\in\mathrm{supp}(\hat{\mu}(dy\,\vert\,\x))$, then asymptotically the growth rate of $\Lambda^{\nu_{\x,t}}_t(y)^{-1}$ is exponential 
as for the CF of the conditional probability $\hat{\mu}(dy\,\vert\,\x)$. To obtain precise asymptotic results when $(\x,y)\in\bom$, additional conditions on $\mu$ are required. 
Below is such a typical result.
\begin{lemma}(Kro\'o and Lubinsky \cite{Lubinsky})
Let $S\subset\R^n$ be compact and assume that
 there exists a measure $\psi_0$ supported on $S$ such that uniformly on compact subsets of $\mathrm{int}(S)$,  $\lim_{t\to\infty}s_n(t)\Lambda^{\psi_0}_t(\x)=W_0(\x)$ where $W_0$ is continuous and positive on $\mathrm{int}(S)$.  

 If a measure $\psi$ has continuous and positive density $D$ w.r.t. $\psi_0$
 on $\mathrm{int}(S)$, then uniformly on compact subsets of $\mathrm{int}(S)$,
 $\lim_{t\to\infty}s_n(t)\Lambda^{\psi}_t(\x)=D(\x)W_0(\x)$. 
 \end{lemma}
 Given a compact set $\mathcal{X}$,
 let $\mathscr{C}(\mathcal{X})$ denote the space of continuous functions on $\mathcal{X}$.
In our context of $\mu$ on a compact set $\bom\subset\X\times\R$ with marginal $\phi$ on $\X$,
we obtain the following result:

\begin{theorem}
\label{th-final}
 Assume that there exists a measure $\mu_0$ on $\bom$ with marginal $\phi_0$ on $\X$
 and conditional $\hat{\mu}_0(dy\,\vert\,\x)$ on $\R$, such that uniformly on compact subsets of $\bom$ (resp. $\X$):
 \[\lim_{t\to\infty} s_{n+1}(t)\,\Lambda^{\mu_0}_t(\x,y)\,=\,W_0(\x,y)\,;\quad
 \lim_{t\to\infty} s_{n}(t)\,\Lambda^{\phi_0}_t(\x)\,=\,W'_0(\x)\,.\]
 In addition assume that the following  Feller-type property holds:
  \[\x\,\mapsto\,\int h(\x,y)\,\hat{\mu}_0(dy\,\vert\,\x)\,\in\,\mathscr{C}(\X)\,\quad\mbox{whenever $h\in \mathscr{C}(\bom)$.}\] 
 Let $\mu$ be a measure on $\bom$ with a continuous and positive density $f$ 
 w.r.t. $\mu_0$. Then, with $\nu_{\x,t}$ being the measure on $\R$  in Theorem \ref{th-main}:
 \begin{equation}
 \label{final}
 \lim_{t\to\infty}\,t\,\Lambda^{\nu_{\x,t}}_t(y)
 \,=\,(n+1)\,\frac{f(\x,y)}{g(\x)}\,\frac{W_0(\x,y)}{W'_0(\x)}\,,\quad\forall (\x,y)\in\mathrm{int}(\bom)\,,
\end{equation}
where $g(\x):=\int f(\x,y)\,\hat{\mu}_0(dy\,\vert\,\x)$. 

Observe that for all $\x\in\mathrm{int}(\X)$, 
$f(\x,y)/g(\x)$ is the density of $\hat{\mu}(dy\,\vert\x)$ w.r.t. $\hat{\mu}_0(dy\,\vert\,\x)$.
 \end{theorem}
\begin{proof}
Disintegrating $\mu_0$ yields $d\mu_0(\x,y)=\hat{\mu}_0(dy\,\vert\,\x)\,\phi_0(d\x)$.
Therefore
\[d\mu(\x,y)\,=\,\frac{f(\x,y)\hat{\mu}_0(dy\,\vert\,\x)}{g(\x)}\,g(\x)\,\phi_0(d\x)\,;\quad 
g(\x)\,:=\,\int f(\x,y)\,\hat{\mu}_0(dy\,\vert\,\x)\,,\]
and $\phi(d\x)\,=\,g(\x)\,\phi_0(d\x)$. Moreover observe that for every $\x\in\X$,
\[\hat{\mu}(dy\,\vert\,\x)\,=\,\frac{f(\x,y)}{g(\x)}\,\hat{\mu}_0(dy\,\vert\,\x)\,.\]
That is, for every $\x\in\X$,  $y\,\mapsto f(\x,y)/g(\x)$ 
is the density of $\hat{\mu}(dy\,\vert\,\x)$ w.r.t. $\hat{\mu}_0(dy\,\vert\,\x)$,
and by the Feller-like property, $g$ is continuous and positive on $\X$.  Next, by our hypotheses and from Theorem \ref{th-main}, 
 \begin{eqnarray*}
 f(\x,y)\,W_0(\x,y)&=&\lim_{t\to\infty}s_{n+1}(t)\Lambda^{\mu}_t(\x,y)\,=\,
 \lim_{t\to\infty}\,\left[s_n(t)\,(\,\Lambda^{\phi}_t(\x)\cdot\frac{s_{n+1}(t)}{s_n(t)}\Lambda^{\nu_{\x,t}}_t(y)\,)\,\right]\\
 &=&g(\x)\,W'_0(\x)\,\lim_{t\to\infty}\frac{s_{n+1}(t)}{s_n(t)}\Lambda^{\nu_{\x,t}}_t(y)\,)\,=\,\frac{g(\x)\,W'_0(\x)}{n+1}\lim_{t\to\infty}t\,\Lambda^{\nu_{\x,t}}_t(y)\,,
 \end{eqnarray*}
 for all 
 $(\x,y)\in\mathrm{int}(\bom)$, which yields \eqref{final}.
\end{proof}
So for instance with $n=1=p$, let $d\mu_0(x,y)=1_{[-1,1]}(x)\,1_{[a(x),b(x)]}(y)\,dx\,dy$
where $x\mapsto a(x)$ and $x\mapsto b(x)$ are strictly positive and continuous on $[-1,1]$, and
$b(x)-a(x)>0$ on $\X=[-1,1]$. Then \eqref{final} reads
\[\lim_{t\to\infty}\,t\,\Lambda^{\nu_{\x,t}}_t(y)
 \,=\,\frac{f(x,y)}{\int_{a(x)}^{b(x)}f(x,y)dy}\,\frac{2\,W_0(x,y)}{W'_0(x)}\,,\quad\forall (x,y)\in\mathrm{int}(\bom)\,,\]
 and $f(x,y)/\int_{a(x)}^{b(x)}f(x,y)dy$ is the density of $\hat{\mu}(dy\,\vert\,x)$ w.r.t. Lebesgue on 
 the interval $[a(x),b(x)]$, for every $\x\in [-1,1]$. 
 
 As expected from the disintegration  \eqref{th-main-1}, convergence  of $t\Lambda^{\nu_{x,t}}_t(y)$ as $t$ increases, is towards
the density of the conditional $\hat{\mu}(dy\,\vert\,x)$ times a weight function intrinsic to
the support $\bom$ of $\mu$, which is typical of convergence results for Christoffel functions
(whenever convergence takes place).

\section{Conclusion}
We have shown that in quite general setup, the Christoffel function disintegrates (or factorizes)
and mimics the disintegration of its associated measure on $\X\times \Y$ into
its marginal on $\X$ and its conditional on $\Y$, given $\x\in\X$. The result uses
a straightforward (but novel) interpretation of a well-known intermediate result of convex optimization,
which is of interest in its own. 
Namely that every SOS polynomial is the reciprocal 
of the Christoffel function associated with some linear functional
(which always has a representing measure in the univariate case). 
A similar interpretation is valid for  the cone of polynomials that are positive on a basic semi-algebraic set.

We think that a better understanding of the linear functional
$\nu_{\x,t}$ (which has  a representing measure when $p=1$) is needed. 
In particular, further investigation beyond the scope of the present note,
could consider a more detailed (and non-asymptotic) 
comparison of $\nu_{\x,t}$ with the conditional $\hat{\mu}(dy\,\vert\,\x)$ when $\x\in\X$, 
as well as understanding its meaning when $p>1$, i.e., when it may not have a representing measure.
For instance we conjecture (but have been unable to prove) that $\nu_{\x,t}$ does \emph{not} depend on $t$,
and has a representing measure.


\begin{thebibliography}{las}
\bibitem{cvx}
Grant M., Boyd S. 
\newblock\emph{CVX: Matlab Software for Disciplined Convex Programming, version 2.1},
\newblock {\tt http://cvxr.com/cvx}, March 2014.
\bibitem{annals-prob}
Helton J.W., Lasserre J.B., Putinar M. \emph{Measures with zeros in the inverse of their moment matrix},
Annals Prob. {\bf 36}, pp. 1453--1471, 2008.
 \bibitem{Lubinsky}
Kro\'{o} A., Lubinsky D.~S. \emph{Christofffel functions and universality in the bulk for multivariate polynomials},
Canad. J. Math. {\bf 65}(3),
pp. 600--620, 2013.
\bibitem{neurips}
Lasserre J.~B., Pauwels E.
\emph{Sorting out typicality via the inverse moment matrix {S}{O}{S} polynomial},
in \emph{Advances in Neural Information Processing Systems},
D.D. Lee, M. Sugiyama, U.V. Luxburg, I. Guyon and R. Garnett Eds., 
Curran Associates, Inc., pp. 190--198, 2016.
\bibitem{CD-2022}
Lasserre J.~B., Pauwels E., Putinar M. \emph{The Christoffel-Darboux Kernel for Data Analysis},
Cambridge Monographs on Applied and Computational Mathematics,
Cambridge University Press, Cambridge, UK, 2022.
\bibitem{adv-comp}
Lasserre J.B, Pauwels E. \emph{The empirical Christoffel  function with applications in data analysis},
Adv. Comput. Math. {\bf  45}, pp. 1439--1468, 2019.
\bibitem{CUP}
Lasserre J.B. \emph{Introduction to Polynomial and Semi-Algebraic Optimization},
 Cambridge University Press, Cambridge, UK, 2015.

\bibitem{nesterov}
Nesterov Y. \emph{Squared functional systems and optimization problems},
in: \emph{High Performance Optimization}, H. Frenk, K. Roos ,T. Terlaky and Shuzong  Zhang (Eds.),
vol 23, Applied Optimization series, Springer, Boston, MA, 2000, pp. 405--440.


\bibitem{FoCM}
Pauwels E., Putinar M., Lasserre J.~B.
\emph{Data analysis from empirical moments and the Christoffel function},
Found. Comput. Math. {\bf 21}, pp. 243--273, 2021.


\end{thebibliography}
\end{document}